\documentclass[12pt,reqno]{amsart}
\usepackage{amssymb,latexsym,amsmath,epsfig,amsthm,mathrsfs}

\usepackage{rotating}
\usepackage{graphicx}
\usepackage{amssymb}
\usepackage{lineno}
\usepackage{enumitem}
\usepackage{cite}
\usepackage{multicol}
\usepackage[top=3cm, bottom=3cm, left=2.5cm, right=2.5cm]{geometry}
\usepackage[usenames]{color}
\usepackage[colorlinks=true,
linkcolor=blue,
filecolor=blue,
citecolor=blue]{hyperref}

\numberwithin{equation}{section}

\newtheorem{proposition}{Proposition}[section]
\newtheorem{corollary}{Corollary}[section]

\newtheorem{question}{Question}[section]

\newtheorem{theorem}{Theorem}[section]
\newtheorem{lemma}[theorem]{Lemma}
\newtheorem{definition}[theorem]{Definition}

\newtheorem{conjecture}[theorem]{Conjecture}

\begin{document}
	
	\title[On  additive complement  with special structures]{On  additive complement  with special structures}
	
	
	\author[Mohan]{Mohan}
	\address{Department of Mathematics, Indian Institute of Technology Roorkee, Uttarakhand, 247667, India}
	\email{mohan98math@gmail.com}


	\author[B R Patil]{Bhuwanesh Rao Patil }
	\address{Department of Mathematics, Indian Institute of Technology Roorkee, Uttarakhand, 247667, India}
	\email{bhuwanesh1989@gmail.com}

	\author[R K Pandey]{Ram Krishna Pandey$^{\ast}$}
	\address{Department of Mathematics, Indian Institute of Technology Roorkee, Uttarakhand, 247667, India}
	\email{ram.pandey@ma.iitr.ac.in}
	\thanks{$^{\ast}$The corresponding author}

	
	
	

		\begin{abstract}
		Let $A$ be a set of natural numbers.	A set $B$, a set of natural numbers, is said to be an additive complement of the set $A$ if all sufficiently large natural numbers can be represented  in the form $x+y$, where $x\in A$ and $y\in B$.  This article describes various types of additive complements of the set $A$ such as those additive complement of $A$  that does not intersects  $A$, additive complements of the form of the union of disjoint infinite arithmetic progressions, additive complement having various density etc. As an application of this study, we also focus on the structure of sumset of arithmetic progression and geometric progression. Apart from this,  for given positive real no. $\alpha \leq 1$ and finite set $A$, we investigate  a set $B$ such that it can be written as union of disjoint infinite arithmetic progression and  density of $A+B$ is $\alpha$.
		
	\end{abstract}
	\maketitle

	\section{Notations}
		Let $\mathbb{N}$, $\mathbb{N}_0$, $\mathbb{Z}$, $\mathbb{Q}$, $\mathbb{R}$ and $\mathbb{P}$ denote the set of all natural numbers, the set of all nonnegative integers, the set of all integers, the set of all rational numbers, the set of all real numbers and the set of all prime numbers, respectively.	For the set $A$, $|A|$ denotes the cardinality of $A$. For positive integers $a$ and $b$, $\gcd(a,b)$ denotes the greatest common divisor of $a$ and $b$.	For given real numbers $a,b$ and a subset $X\subset \mathbb{R}$, we denote
		$$(a,b)_{X}:=\{x\in X: a<x<b \}, \quad
		(a,b]_{X}:=\{x\in X: a<x\leq b \},$$
		$$[a,b)_{X}:=\{x\in X: a\leq x<b \}, \quad
		[a,b]_{X}:=\{x\in X: a\leq x\leq b \},$$
		$$\text{and} \quad [a,b]:=[a,b]_{\mathbb{N}}.$$
		For given  subsets $A$ and $B$ of $\mathbb{N}$ and an integer $u$, we define $$A+B=\{x+y:x\in A,y\in B\}, \quad \text{} \quad A-B=\{x-y:x\in A, y\in B \},$$
		$$u+B=\{u+y:y\in B\}, \quad  u-B=\{u-y:y\in B\}, \quad \text{and} \quad uB=\{uy: y\in B \}.$$

	\section{Introduction}

	Representing  a natural number as a sum of two natural numbers belonging to some prescribed sets is an interesting problem among many mathematicians. In precise, 
	for given subsets $A$ and $B$ of  $\mathbb{N}$, can we write all sufficiently large natural numbers in the form of $x+y$ where $x\in A$ and $y\in B$? For example, all sufficiently large natural numbers can be written as sum of a prime number and a composite number, but infinitely many natural numbers can not be expressed as the sum of two prime numbers. This motivates us to think that for given set $A$, there is a set $B$ such that we get a positive answer  to the above question and there is a set $B$ such that we get negative answer to the above question. This leads to the following question.
	\begin{question}
		For given set $A\subset
		\mathbb{N}$, classify all those subsets $B$ of $\mathbb{N}$ such that all sufficiently large natural numbers  can be written in the form of $x+y$ where $x\in A$ and $y\in B$?
	\end{question} 
	\noindent This boils down to the following notion of {\it{additive complement}}:
	\begin{definition}[Additive complement\footnote{Note that the `complement of a set' is different from `additive complement of a set'. Complement of a set $A$ in $\mathbb{N}$ is all natural number not belonging to $A$.}]
		Let $A\subset \mathbb{N}$. Then a set of natural numbers $B$ is said to be an additive complement of $A$ if $(\mathbb{N}\setminus (A+B))$ is finite.
	\end{definition}
	Note that the set $\mathbb{N}$ is always  an additive complement of every subset of $\mathbb{N}$. Now we are looking for some sparse additive complement for a given set. To get a notion of sparseness, we consider the notion of \textit{density} of a set. 
	\begin{definition}
		Let $A$ be a set of natural numbers. Then upper asymptotic density and lower asymptotic density, denoted by $\overline{d}(A)$ and $\underline{d}(A)$ respectively, is defined as follows
		\[\overline{d}(A) = \limsup_{n\rightarrow \infty} \dfrac{|A \cap [1,n]|}{n} \quad \text{and} \qquad \underline{d}(A) = \liminf_{n\rightarrow \infty} \dfrac{|A \cap [1,n]|}{n}. \]
		If $\overline{d}(A) = \underline{d}(A) $, then we say that density of $A$, denoted by $d(A)$, exists  where $\overline{d}(A) = \underline{d}(A)  = d(A)$. 
	\end{definition}
	\noindent Using this notion, we define that a set $A$ is said to be large if the $\overline{d}(A) > 0$ and a set is said to be sparse if it is not a large set. In communication with G.G. Lorentz, P. Erd\H{o}s conjectured that for each infinite set  of  natural numbers there exist an additive complement  with density zero \cite{Lorentz}. In 1954, Lorentz gave affirmative solution to this conjecture in the form of following proposition. 
	\begin{proposition}[Lorentz, \cite{Lorentz}]\label{Prop-Lorentz} Let $A$ be an infinite set of natural numbers. Then there exists a subset $B$ of $\mathbb{N}$ such that  $B$ is an additive complement of $A$ and \begin{equation*}
			\left |B \cap [1,n] \right | \leq C\sum_{k=1}^{n} \dfrac{\log |A \cap [1,k]|}{|A \cap [1,k]|},
		\end{equation*}
		where $C$ is an absolute constant and the terms of the sum with $|A \cap [1,k]| = 0$ are to be replaced by one.
	\end{proposition}
	\noindent This proposition guarantees the existence of sparse additive complement of an infinite set in $\mathbb{N}$, but one would be interested for  those additive complement of an infinite set that is subset of  a given set of natural numbers.  By additive complements of a set $A$ in given set $C$, we mean those additive complements of $A$ that are subsets of $C$.   More precisely, we ask the following question.
	\begin{question}\label{q2}
		For given $A \subset \mathbb{N}$ and $C\subset \mathbb{N}$, does there exist an additive complement of $A$ in $C$?
	\end{question}
	\noindent Trivially, answer of this question is negative by observing that there is no additive complement of the set of even natural numbers in the set of  odd natural numbers. Moreover, Theorem \ref{Thm-1}, given below, gives infinite  family of sets  $A$ such that every additive complement of $A$ intersects set $A$.
	\begin{theorem}\label{Thm-1}
		Let $S$ be an infinite subset of the natural numbers. Then there exists $A\subset \mathbb{N}$  such that $(A+(\mathbb{N}\setminus A))\cap S=\varnothing.$ 
	\end{theorem}
	\noindent
	Proof of this theorem is given in Section \ref{section3}. By investigating set $A$ with given density $\alpha$ in the above theorem, we obtain the following theorem.
	\begin{theorem}\label{Thm-2}
		Let $\alpha\in [0,1]_{\mathbb{R}}$. Then there exists $A\subset \mathbb{N}$ with $d(A)=\alpha$ such that there is no additive complements of $A$ in the complement of $A$. 
	\end{theorem}
	\noindent
	Proof of this theorem is given in Section \ref{section4}. Existence  of density zero set in conclusion of the above theorem can be obtained by Theorem \ref{Thm-1} (see Corollary \ref{not_additive_complement_density_zero}).  But to  construct a set $A$   with given   positive density satisfying conclusion of the above theorem,  we investigate a set which is the union of disjoint infinite arithmetic progressions  of given density and that leads to the following question. 
	\begin{question}
		For given $\alpha \in (0,1]_{\mathbb{R}}$ and a finite set $B \subset \mathbb{N}$, does there exist a set $A \subset\mathbb{N}$ such that $d(A+B) = \alpha$ and $A$ is disjoint union of infinite arithmetic progressions? 
	\end{question}
	\noindent In the above question, investigation of  $A$ as the union of finite arithmetic progressions instead of the union of disjoint infinite arithmetic progressions is settled by Grekos et. al. \cite{Grekos}.  Some work related to the above question is also done  in articles Erdos \cite{erdos}, Lorentz \cite{Lorentz}, Faisant et. al \cite{Faisant}.  The following theorem settle the above question.  
	\begin{theorem}\label{Thm-3}
		Let $B$ be a finite subset of natural numbers and $\alpha \in (0,1]_{\mathbb{R}}$. Then the following statements hold:
		\begin{enumerate}[label=(\alph*)]
			\item   \label{THM-3Part1}  There exists a set $A\subset \mathbb{N}$ such that $A$ is the union of disjoint infinite arithmetic progressions and $d(A+B)=\alpha$.
			\item \label{THM-3Part2}There exists a set $A$ such that  $d(A)=\alpha$, it is the union of disjoint infinite arithmetic progressions and it also contains infinitely many primes.
			\item  \label{THM-3Part3}There exists a set $A$,  such that $d(A)=\alpha$, it is the union of disjoint infinite arithmetic progressions and it does not contain any prime.
		\end{enumerate} 
	\end{theorem}
	\noindent Proof of the above theorems is given in Section \ref{section5}.
	%
	
	In order  to study   Question \ref{q2}, we are able to  answer the  following particular question in this article:
	\begin{question}\label{q3}
		Let $B$ be an infinite set of natural numbers  and $\alpha\in (0,1]_{\mathbb{R}}$. Does there exist a set of natural numbers $E$ such that $d(E)=\alpha$, $E$ is an additive complement of $B$, and $E$ is the union of disjoint infinite arithmetic progressions?
	\end{question} 
	If we are able to produce the union of disjoint infinite arithmetic progressions, with arbitrary density, which contains a given set with zero density, then Proposition \ref{Prop-Lorentz} gives an affirmative answer of the above question. For example, one can produce the union of disjoint infinite arithmetic progressions, with arbitrary density, which contains all primes, by using Theorem \ref{Thm-3}. But, we are not sure whether any set with zero density can be seen inside the
	union of disjoint infinite arithmetic progressions with arbitrary density.
	So we are answering  Question \ref{q3} here by some different method. 
	For the case of infinite geometric progression, we may get an affirmative answer of Question \ref{q3} using the theory of primitive root  in the form of Theorem \ref{Thm-4}.

	\begin{theorem}\label{Thm-4}
		Let $\alpha \in (0,1]_{\mathbb{R}}$, $a\in \mathbb{N}$, $g\in \mathbb{N}\setminus \{1\}$, and $B$ be the set of all elements of the form $ag^k$ with $k\in \mathbb{N}$. Suppose that there exists prime $p$ such that $p\nmid a$ and $g$ is primitive root modulo $p^2$. Then there exists an additive complement $A$ such that it can be written as the union of disjoint infinite arithmetic progressions and  $d(A)=\alpha$.
		
	\end{theorem}
	\noindent
	Using the fact that there exist  infinitely many primes $q$ such that $q$ is primitive root modulo $p$ for infinitely many prime $p$ (see Corollary \ref{primitive_root_heath_boron_2} and Definition \ref{primitive_root_definition}), we obtain the following corollary of the above theorem.
	\begin{corollary} \label{THM-4Cor1}
		There are infinitely many infinite geometric progressions such that the conclusion of Theorem \ref{Thm-4} holds.
	\end{corollary}
	\noindent 
	Proof of Theorem \ref{Thm-4} and the above corollary is given in Section \ref{section6}.

	\section{Preliminary}\label{section2}
	This section provides some information required to prove main theorems. 
	\subsection{Primitive root}
	\begin{definition}\label{primitive_root_definition}
		Let $n$ and $X$ be positive integers. Then  $x$ is called primitive root modulo $n$ if for every integer $a$  relative coprime to $n$, there exist positive integer $m_a$ such that   $x^{m_a}\equiv a\pmod{n}.$
	\end{definition}
	In relation to primitive root modulo prime, Artin gave the following conjecture:
	\begin{conjecture}
		Let $a$ be integer such that $a$ is neither perfect square nor $-1$. Then there exists infinitely many primes $p$ such that $a$ is primitive root modulo $p$.
	\end{conjecture}
	\noindent
	Yet, this conjecture is not resolved for any single value. Heath-Brown \cite{Heath-Brown} has given some partial answer in the positive direction of Artin's conjecture in the form of Lemma \ref{primitive_root_heath_boron_1} and Corollary \ref{primitive_root_heath_boron_2}.
	\begin{lemma}[Heath-Brown, \cite{Heath-Brown}] \label{primitive_root_heath_boron_1}
		Let $p, q$ and $r$ be multiplicatively independent integers\footnote{The three non zero integers $p,q$ and $r$ are multiplicatively independent integers if $p^aq^br^c=1$ with $a,b,c\in \mathbb{Z}$ implies $a=0,b=0$, and $c=0$.} such that none of $q,r,s,-3qr$, $-3qs,-3rs$ or $qrs$ is a square. Then there exists infinitely many primes $p$ such that at least one of $q,r$ and $s$ is primitive root modulo prime $p$.
	\end{lemma}
	\begin{corollary}[Heath-Brown, \cite{Heath-Brown}]\label{primitive_root_heath_boron_2}
		There are at most two primes for which artin conjecture does not hold.
	\end{corollary}
	The following two lemmas describe primitive roots modulo prime power.
	\begin{lemma}\label{primitilive_root_property1}
		If $p$ is an odd prime and  $g$ is primitive root modulo $p^2$, then $g$ is primitive root mod $p^k$ for each positive integer $k\geq 3$.
	\end{lemma}
	\begin{lemma}\label{primitive_root_property2}
		Let $g$ be a natural number and $p$ be  a prime such that $g$ is   primitive root modulo $p$ and $g$ is not  primitive root modulo $p^2$. Then $g+tp$ is primitive root modulo $p^2$ for every $t$ with $t\not \equiv 0\pmod{p}$.
	\end{lemma}
	\noindent
	Proof of the above two lemmas is given in the book \cite[Page 102]{niven}.
	
	\subsection{Base representation of real number}
	\begin{lemma}[\cite{carather}, page 8]\label{thm_lem_1}
		Let $\alpha\in [0,1]_{\mathbb{R}}$ and $q$ be an integer such that $q\geq 2$. Then there exists a sequence of integers $(a_n)_{n=1}^{\infty}$ with $0\leq a_n\leq q-1$ for each $n$ such that
		$\alpha =\sum_{n=1}^{\infty}\frac{a_n}{q^n}$.
	\end{lemma}
	\subsection{Euler's totient function}
	\begin{definition}\label{totient_function}
		For every $n\in \mathbb{N}$, $\phi(n)$ denotes the cardinality of the set $\{x\in [1,n]: \gcd(x,n)=1\}$. Then  the map $\phi:\mathbb{N} \rightarrow \mathbb{N}$ is called Euler's totient function.
	\end{definition}
	The following property of totient function $\phi$ can be obtained using the fact observed in \cite[Lemma 2.5, page 6]{totient}.
	\begin{lemma}\label{totient_property1}
		$\displaystyle\liminf_{n\rightarrow \infty}\dfrac{\phi(n)}{n}=0$
	\end{lemma}
	\subsection{Union of disjoint arithmetic progressions}
	\begin{definition}[Arithmetic progression]
		Let $n\in \mathbb{N}$. Then the sequence $(a_i)_{i=0}^{n-1}$ in $\mathbb{N}$ is called $n$-term arithmetic progression if $a_i=x+id$ for some positive integer $x$ and  integer $d$. Also the sequence $(a_i)_{i=0}^{\infty}$ in $\mathbb{N}$ is called infinite arithmetic progression if $a_i=x+id$ for some positive integer $x$ and  integer $d$.
	\end{definition}
	\begin{definition}[UDAP set]
		Let $A$ be a nonempty set of natural numbers. Then the set $A$  is said to be an UDAP set if it can be written as the union of disjoint infinite arithmetic progressions.
	\end{definition}
	It is  easy to check  that an infinite arithmetic progressions is UDAP set but the union of an  infinite  arithmetic progression and a finite set of integers may not be an UDAP set. In our proofs of main theorems, we are using a special type of UDAP set which is described below in the form of $	\mathbb{M}(q,(a_j)_{j=1}^{\infty})$ for $q\in \mathbb{N}\setminus\{1\}$ and sequence $(a_j)_{j=1}^{\infty}$ in $[0,q-1]$:
	
	\begin{definition}\label{main_UDAP_set}
		Let $q\in \mathbb{N}\setminus\{1\}$,   and $(a_n)_{n=1}^{\infty}$ be a sequence in $[0,q-1]$. Then the set $\mathbb{M}(q,(a_j)_{j=1}^{\infty})$ is defined in the following way:
		\begin{equation*}
			\mathbb{M}(q,(a_j)_{j=1}^{\infty}):=\bigcup_{i=1}^{\infty}(X_i\setminus \{0\}), \text{ where }   	X_n=\bigcup_{i=0}^{a_n-1} \left(\sum_{j=0}^{n-2} a_{j+1}q^j+ iq^{n-1} +q^n\mathbb{N}_0\right) ~\forall~ n\in \mathbb{N}
		\end{equation*} 
		
	\end{definition}
	\noindent
	The following lemma guarantees that density of the set of the form $\mathbb{M}(q,(a_j)_{j=1}^{\infty})$ exists.
	\begin{lemma}\label{density_union_of_ap}
		Let $q\in \mathbb{N}\setminus \{1\}$ and $(a_i)_{i=1}^{\infty}$ be a sequence in $[0,q-1]$. Let $(X_i)_{i=1}^{\infty}$ be a sequence of infinite subsets of  $\mathbb{N}$ such that $X_i$ is union of $a_i$ disjoint infinite arithmetic progressions having common difference $q^i$ for each $i\in \mathbb{N}$,    $X_{j}\cap X_i=\varnothing$ for $i\neq j$, and  $\bigcup_{i=m+1}^{\infty}X_{i}$ is contained in an arithmetic progression having common differences $q^m$ for each $m\in \mathbb{N}$. Then density of the set $\bigcup_{i=1}^{\infty}X_i$ exists and equal to $\sum_{i=1}^{\infty}\frac{a_i}{q^i}$.
	\end{lemma}
	\begin{proof}
		Let $A=\bigcup_{i=1}^{\infty}X_i$. Since $X_i$ is union of $a_i$ disjoint infinite arithmetic progressions having common difference $q^i$, we get that 
		$d(X_i)=\frac{a_i}{q^i}$ for every  positive integer $n$.
		Using the fact that $X_i$'s are pairwise disjoint, and $\bigcup_{i=m+1}^{\infty}X_{i}$ is contained in an arithmetic progression having common differences $q^n$ for each $n\in \mathbb{N}$, we have
		$$\overline{d}(A)=\overline{d}\left(\bigcup_{i=1}^{\infty}X_i\right)\leq \overline{d}\left(\displaystyle\bigcup_{i=1}^{n}X_i\right)+\overline{d}\left(\displaystyle\bigcup_{i=n+1}^{\infty}X_i\right)\leq \sum_{i=1}^{n}\frac{a_i}{q^i}+ \frac{1}{q^n}, ~\forall~ n\in \mathbb{N} $$
		and
		$$\underline{d}(A)=\underline{d}\left(\bigcup_{i=1}^{\infty}X_i\right)\geq \underline{d}\left(\displaystyle\bigcup_{i=1}^{n}X_i\right)+\underline{d}\left(\displaystyle\bigcup_{i=n+1}^{\infty}X_i\right)\geq \sum_{i=1}^{n}\frac{a_i}{q^i} ~\forall~ n\in \mathbb{N}.  $$
		Combining the above two equations, we get that for every positive integers $n$,
		$$\sum_{i=1}^{n}\frac{a_i}{q^i}\leq \underline{d}(A)\leq \overline{d}(A)\leq  \sum_{i=1}^{n}\frac{a_i}{q^i}+ \frac{1}{q^n}.$$
		Thus, by taking limit as $n\rightarrow \infty$ in this equation, we obtain that
		$	d(A)=\sum_{i=1}^{\infty}\frac{a_i}{q^i}$.
	\end{proof}
	\noindent
	By using the above lemma, we get the density of the set of the forms $\mathbb{M}(q,(a_j)_{j=1}^{\infty})$, which is given in the corollary given below.
	\begin{corollary}\label{density_union_ap_corollary}
		Let $q\in \mathbb{N}\setminus\{1\}$   and $(a_n)_{n=1}^{\infty}$ be a sequence in $[0,q-1]$. Then density of the set $ \mathbb{M}(q,(a_j)_{j=1}^{\infty})$ exists and equals to $\sum_{i=1}^{\infty}\frac{a_i}{q^i}$.
	\end{corollary}
	\subsection{Primes in arithmetic progresssions}
	\begin{lemma}[Dirichlet's Theorem  \cite{hardy}]\label{dirichlet}
		Let $a$ and $b$ be positive integers with no common divisor except 1. Then there are infinitely many primes of
		the form $an+b$.
		
	\end{lemma}

	\section{Proof of Theorem \ref{Thm-1}}\label{section3}
	\begin{proof}[Proof of Theorem \ref{Thm-1}]
		Let $S=\{b_i: i\in \mathbb{N}\}$ such that $b_i<b_{i+1}$ for $i\in \mathbb{N}$. Let $X$ be any arbitrary subset of $\mathbb{N}$. We construct  set $A$ in the following way:
		\begin{align*}
			A_1&= (X\cup (b_1-X) )\cap \mathbb{N}\\
			A_n&=\left(\{b_1,b_2,\cdots, b_n\}-\left(\bigcup_{i=1}^{n-1} A_i\right)\right)\cap \mathbb{N} ~\forall~ n\in \mathbb{N}\\
			A&=\bigcup_{i=1}^{\infty}A_i.
		\end{align*}
		
		To prove $(A+(\mathbb{N}\setminus A))\cap S=\emptyset$, it is sufficient to prove that if $b_{i}-x >0$ for $x \in A$ and $i\in \mathbb{N}$, then $b_{i}-x \in A$. Proof of this given below.
		
		Assume  $x\in A$ and $i\in \mathbb{N}$ such that $b_{i}-x>0$. Then there exists $j\in \mathbb{N}$ such that $x\in A_j$. This implies that $b_i-x\in A_{l+1}$ where $l=\max\{i,j\}$. Therefore $b_i-x\in A$. This completes the proof the the theorem.
	\end{proof}
	\noindent 
	As a corollary to Theorem \ref{Thm-1}, we obtain  a set $A$ with density zero so that $A+(\mathbb{N}\setminus A)$ misses a geometric progression. Corollary \ref{not_additive_complement_density_zero} gives detail regarding this observation.
	\begin{corollary}\label{not_additive_complement_density_zero}
		Let $g\in \mathbb{N}\setminus \{1\}$ and  $S$ be the set of all elements of the form $g^i$ with $i\in \mathbb{N}$. Then there exists a set $A$, a subset of $\mathbb{N}$, such that $d(A)=0$ and  $(A+(\mathbb{N}\setminus A))\cap S=\varnothing$.
	\end{corollary}
	\begin{proof}
		Let $g$ be a positive integer such that  $g\geq 3$. Let $S=\{g^i:i\in \mathbb{N}\}$ and $X=\{1\}$. Using the construction in the proof of Theorem \ref{Thm-1}, we get the following set $A$ such that $A+\mathbb{N}\setminus A$ does not contain  the infinite set $S$:
		\begin{align*}
			A_1&=X\cup (\{g\}-X)\cap \mathbb{N}=\{1,g-1\},\\
			A_n&=\left(\{g,g^2,\cdots, g^n\}-\bigcup_{i=1}^{n-1}A_i\right) \cap \mathbb{N}~ \forall ~  n\geq2,\\
			A&=\bigcup_{i=1}^{\infty}A_i.
		\end{align*}
		Note that \begin{equation}
			A_n\subset A_{n+1} ~\forall ~n\in \mathbb{N}.\label{Ex_densityzero_not_disjoint_additive_complement_eq0}
		\end{equation} We also
		observe that  for each $n\in \mathbb{N}$, $A_n$ and $(g^{n+1}-A_n)$ are disjoint because $\max(A_n)=g^n-1<\min(g^{n+1}-A_n)$. By induction, we shall now show that  \begin{equation}\label{Ex_densityzero_not_disjoint_additive_complement_eq1}
			A_n= A_{n-1}\cup (g^n-A_{n-1}) ~\forall~ n\geq2.
		\end{equation}
		
		It is easy to observe that the equation \eqref{Ex_densityzero_not_disjoint_additive_complement_eq1} holds for $n=2$.
		Suppose that the equation \eqref{Ex_densityzero_not_disjoint_additive_complement_eq1} holds for each $1<n\leq k$. To show the validity of the equation \eqref{Ex_densityzero_not_disjoint_additive_complement_eq1} for each natural number $n>1$, using the way of induction, we need to show that
		$$A_{k+1}= A_{k}\cup (g^{k+1}-A_{k}).$$
		
		Now applying  the expression \eqref{Ex_densityzero_not_disjoint_additive_complement_eq0} in the definition of $A_{k+1}$, we get
		\begin{align*}
			A_{k+1} & = (\{g,g^2,\cdots, g^{k+1}\}-A_k)\cap \mathbb{N}\\
			& = \left((\{g,g^2,\cdots, g^{k}\}-A_{k-1})\cup (\{g,g^2,\cdots, g^{k}\}-(A_k\setminus A_{k-1})) \cup ( g^{k+1}-A_k)\right)\cap \mathbb{N}.
		\end{align*}
		Since $A_k$ and $(g^{k+1}-A_k)$ are disjoint, we get the following equation by  applying induction hypothesis in the above equation.
		\begin{align*}
			A_{k+1} &=  (A_k \cup (\{g,g^2,\cdots, g^{k}\}-(g^k-A_{k-1}))\cup (g^{k+1}-A_k))\cap \mathbb{N}\\
			& =  A_k \cup (g^{k+1}-A_k) \cup \left(\left(\bigcup_{i=1}^{k-1}(g^i-g^k+A_{k-1})\right)\cap \mathbb{N}\right)\\
			& =  A_k \cup (g^{k+1}-A_k),
		\end{align*}
		because $\max(g^i-g^k+A_{k-1})=g^i-g^k+max(A_{k-1})<g^i-g^k+g^{k-1}<0$ for each $i\leq k-1$.
		Thus we have shown the equation \eqref{Ex_densityzero_not_disjoint_additive_complement_eq1}.
		
		The equation \eqref{Ex_densityzero_not_disjoint_additive_complement_eq1} gives us $|A_n|=2 |A_{n-1}|$, because $A_{n-1}$ and  $(g^n-A_{n-1})$ are disjoint for each $n\geq 2$. It follows that $|A_n|=2^n$ for each $n\in \mathbb{N}$.

		Since $g\geq 3$, we get that for each natural number $n$
		$$|A\cap[1,n]|\leq  |A_{\lfloor \log_g(n)\rfloor+1}|\leq  2^{1+\log_g(n)}\leq  2(3^{\frac{2\log_3(n)}{3}})\leq 2n^{\frac{2}{3}}. $$
		Therefore,
		$$d(A)=\lim_{n\rightarrow \infty}\dfrac{|A\cap[1,n]|}{n}\leq  \lim_{n\rightarrow \infty} \dfrac{2n^{\frac{2}{3}}}{n}=0. $$
	\end{proof}
	Note that the above corollary is also helpful in the proof of Theorem \ref{Thm-2}, which is given in the next section.
	\section{Proof of Theorem \ref{Thm-2} }\label{section4}
	As a consequence to Corollary \ref{not_additive_complement_density_zero}, we get the proof of Theorem \ref{Thm-2} in the case $\alpha=0$ and $\alpha=1$.
	
	If $\alpha $ is rational, we get the proof of Theorem \ref{Thm-2} using the following proposition, namely Proposition \ref{THM-4rational_proposition}.		\begin{proposition}\label{THM-4rational_proposition}
		Let $\alpha\in (0,1]_{\mathbb{Q}}$. Then there exists $A\subset \mathbb{N}$  such that $d(A)=\alpha$, and $(A+(\mathbb{N}\setminus A))$ does not intersect some infinite subset of $\mathbb{N}$.
	\end{proposition} 		
	\begin{proof}
		Let $\alpha=\frac{p}{q}$ such that  $p<q$ and $\{p,q\}\subset \mathbb{N}$. Let $A\subset\mathbb{N}$ such that
		$A:=\left( \bigcup_{i\in [0,p-1]} i+q\mathbb{N}_{0}\right)\cap \mathbb{N}.$
		It is easy to observe that $d(A)=\frac{p}{q}$ and
		$A+(\mathbb{N}\setminus A)=\mathbb{N}\setminus (p-1+q\mathbb{N}_0)$.
		This completes the proof of the proposition.	
	\end{proof}
	\noindent	The next proposition completes the proof of Theorem \ref{Thm-2}.
	\begin{proposition}\label{existence_non_additive_complement_density}
		Let $\alpha\in (0,1)_{\mathbb{R}}$. Then there exists $A\subset \mathbb{N}$   such that $d(A)=\alpha$, and $A+(\mathbb{N}\setminus A)$ does not intersect some infinite subset of $\mathbb{N}$. 	
	\end{proposition}
	\begin{proof}
		One can easily observe that $d(A)=\alpha$ if and only if $d(\mathbb{N}\setminus A)=1-\alpha$. So it is sufficient to prove the proposition for $\alpha\geq \frac{1}{2}$.

		Let $\alpha \geq \frac{1}{2}$. By Lemma \ref{thm_lem_1}, there exists a sequence of integers $(a_i)_{i=1}^{\infty}$ in $\{0,1\}$ with $a_1=1$ such that
		$$\alpha =\sum_{i=1}^{\infty}\frac{a_i}{2^i}.$$
		
		Let $A\subset \mathbb{N}$ such that $A=\mathbb{M}(2,(a_j)_{j=1}^{\infty})$.	By using Corollary \ref{density_union_ap_corollary}, we get that $d(A)=\alpha$. By Definition \ref{main_UDAP_set}, we can write the set $A$ and the set $\mathbb{N}\setminus A$ in the following form:
		$$ A =\bigcup_{i=1}^{\infty}(X_i\setminus\{0\}), ~\text{where}~ X_n=\bigcup_{i=0}^{a_n-1} \left(\sum_{j=0}^{n-2} a_{j+1}2^j+ i2^{n-1} +2^n\mathbb{N}_0\right) ~\forall n\in \mathbb{N}.$$
		$$\mathbb{N}\setminus A=\bigcup_{i=1}^{\infty} (Y_i\setminus\{0\}), ~\text{where}~ Y_n=\bigcup_{i=a_n+1}^{1}\left(\sum_{j=0}^{n-2} a_{j+1}2^j+ i2^{n-1} +2^n\mathbb{N}_0    \right) \forall ~n\in \mathbb{N}.$$

		Since $\alpha\in (0,1)_\mathbb{R}$, then the set $\{i\in \mathbb{N}: a_i=0\}$ is nonempty. So, By Well-ordering-principle\footnote{Every nonempty set of positive integers contains least element.}, $w:=\min\{i\in \mathbb{N}: a_i=0\}$ exists. This means $a_i=1$ for each $i\leq w$. Since $a_1=1$, we have $w>1$.

		
		Next,  we  compute the set  $A+(\mathbb{N}\setminus A)$. To get this,
		let $x\in A$ and $y\in \mathbb{N}\setminus A$.  
		Since $y\in \mathbb{N}\setminus A$, then there exists $m\in \mathbb{N}$ with $a_m\neq1$ such that $y\in Y_m$. Then  we have 
		\begin{equation}\label{existence_non_additive_complement_density_eq1}
			y\in \sum_{j=0}^{m-2} a_{j+1}2^j+ 2^{m-1} +2^m\mathbb{N}_0.
		\end{equation}
		Since $x\in A$, there exists $l\in \mathbb{N}$ with $a_l\neq 0$ such that
		$x\in X_l$. Then  we have 
		\begin{equation}\label{existence_non_additive_complement_density_eq2}
			x\in \sum_{j=0}^{l-2} a_{j+1}2^j +2^l\mathbb{N}_0.
		\end{equation}
		Using the expressions \eqref{existence_non_additive_complement_density_eq1}
		and \eqref{existence_non_additive_complement_density_eq2}, we obtain
		\begin{equation*}
			x+y\in \begin{cases}
				\displaystyle\sum_{j=1}^{l-1}a_{j}2^{j}+2^{l-1} +2^l\mathbb{N}_0 &\text{ if } l<m \\ 
				\displaystyle\sum_{j=1}^{m-1}a_{j}2^{j}+2^{m-1} +2^m\mathbb{N}_0 &\text{ if } l>m.\\
			\end{cases}
		\end{equation*}
		Thus,
		$$A+(\mathbb{N}\setminus A)\subset \bigcup_{l=1}^{\infty} T_l,$$
		where $T_l$ is defined as $$T_l:=\displaystyle\sum_{j=1}^{l-1}a_{j}2^{j}+2^{l-1} +2^l\mathbb{N}_0. $$
		
		Define
		$$S:=\left\{\sum_{j=1}^{l-1}a_j2^j:l\in \mathbb{N}\setminus\{1\}\right\}.$$
		To complete the proof, we need to show that  $S$ does not intersect the set $A+(\mathbb{N}\setminus A)$. For this, it is enough to show that
		$S$ and  $T_k$ are disjoint for every positive integer  $k$. The proof of this is given below.
		
		By the way of contradiction, if possible, we assume that $k$ is a positive integer such that $S\cap T_k\neq \varnothing$. So there exists positive integer $u$ such that  $$\sum_{j=1}^{u-1}a_j2^j\in S\cap T_k .$$ Then the definition of $T_k$ gives that
		\begin{equation}
			\sum_{j=1}^{u-1}a_j2^j=\left(\sum_{j=1}^{k-1}a_j2^j\right)+2^{k-1}+2^kw,\label{existence_non_additive_complement_density_eq3}
		\end{equation}
		for some nonngative integer  $w$. Now we consider the following two cases:
		\begin{enumerate}[label=(\alph*)]
			\item If $k\geq u$, the equation \eqref{existence_non_additive_complement_density_eq3} is not correct because right hand  side of equation \eqref{existence_non_additive_complement_density_eq3} is greater than  left hand  side of equation \eqref{existence_non_additive_complement_density_eq3}. Thus we get a contradiction for the case $k\geq u$.
			\item If $k< u$,  then equation \eqref{existence_non_additive_complement_density_eq3} gives that
			$2^{k-1}\equiv 0 \pmod{2^k}$, which is not valid equation for any integer $k$. So we also get contradiction for the case $k>u$. 
		\end{enumerate} 
		Thus $S\cap T_k=\varnothing$. This completes the proof. 
	\end{proof}

	\section{Proof of Theorem \ref{Thm-3} }\label{section5}
	\begin{lemma}\label{existence_given_density_set}
		Let $\alpha\in (0,1]_{\mathbb{R}}$,  $k\in \mathbb{N}$ and $B\subset [0,k-1]$ such that $\min(B)=0$ and $\max(B)=k-1$. Then there exists $A\subset \mathbb{N}$ such that $d(A+B)=\alpha$ and $A$ is the union of disjoint infinite arithmetic progression.
	\end{lemma}
	\begin{proof}
		Let $q$ be a positive integer such that $q\geq 2$. Then the sequence $(q^n\alpha)_{n=1}^{\infty}$ diverges to infinity. So, we can choose a  positive integer $r$ and nonnegative integer $c$ such that $c=\lfloor q^r\alpha\rfloor$ and
		\begin{equation}\label{existence_given_density_set_eq_3}
			2k-2\leq c.
		\end{equation}
		Since $\alpha \in (0,1]_\mathbb{R}$, by Lemma \ref{thm_lem_1}, there exists a sequence of integers $(a_i)_{i=1}^{\infty}$ with $0\leq a_n\leq q-1$ for each $n\in \mathbb{N}$ such that
		\begin{equation}\label{existence_given_density_set_eq_2}
			\alpha =\sum_{i=1}^{\infty}\frac{a_i}{q^i}= \frac{c}{q^r}+ \sum_{i=r+1}^{\infty}\frac{a_i}{q^i}.
		\end{equation}

		Define
		$A:=X_r \cup( c-k+1+\mathbb{M}(q,(b_i)_{i=1}^{\infty}))$, where
		$$X_r:=\bigcup_{i=0}^{c-k} ((i+q^r\mathbb{N}_0)\setminus\{0\})\text{ and } 
		b_i= \begin{cases}
			0 &\text{ if } i\leq r\\
			a_i &\text{ if } i>r.
		\end{cases}
		$$
		
		By the construction of the set $A$, one can easily observe that $A$ is UDAP set, because  the UDAP set $c-k+1 +\mathbb{M}(q,(b_i)_{i=1}^{\infty})$ and  the UDAP set $X_r$ are disjoint.

		Now we compute the quantity $d(A+B)$. Observe that
		\begin{align}
			A+B=&  (X_r+B)\bigcup \left(B+c-k+1+\mathbb{M}(q,(b_i)_{i=1}^{\infty})\right).\label{existence_given_density_set_eq_5}
		\end{align}
		It is given in the assumptions   that $B\subset [0,k-1]$ such that $\min(B)=\{0\}$ and $\max(B)=k-1$. This gives us
		\begin{equation*}
			X_r+B=X_r\bigcup\left(X_r+k-1\right)\bigcup \left(\bigcup_{i\in B\cap [1,k-2]}X_r+i\right).
		\end{equation*}
		Using the expression \eqref{existence_given_density_set_eq_3} in the above equation, we get that
		\begin{equation}
			X_r+B=\bigcup_{i=0}^{c-1}((i+q^r\mathbb{N}_0)\setminus \{0\})\label{existence_given_density_set_eq_8}
		\end{equation}
		By Definition \ref{main_UDAP_set}, $\mathbb{M}(q,(b_i)_{i=1}^{\infty})\subset q^r\mathbb{N}$. This gives us 	
		\begin{equation}
			([0,k-2]\cap B)+c-k+1+\mathbb{M}(q,(b_i)_{i=1}^{\infty})\subset \bigcup_{i=0}^{c-1}((i+q^r\mathbb{N}_0)\setminus\{0\})\label{existence_given_density_set_eq_9}
		\end{equation}
		Combining equations \eqref{existence_given_density_set_eq_5}, \eqref{existence_given_density_set_eq_8} and \eqref{existence_given_density_set_eq_9}, we get that
		\begin{equation}
			A+B
			=\left(\bigcup_{i=0}^{c-1}(i+q^r\mathbb{N})\setminus\{0\}\right)\bigcup  \left(c+\mathbb{M}(q,(b_i)_{i=1}^{\infty}\right).\label{existence_given_density_set_eq_6}
		\end{equation}
		By Corollary \ref{density_union_ap_corollary}, density of the set $\mathbb{M}(q,(b_i)_{i=1}^{\infty})$ exists and equal to $\sum_{i=r+1}^{\infty} \frac{a_i}{q^i}$. So, translation invariance property of density gives
		\begin{equation}
			d(c+\mathbb{M}(q,(b_i)_{i=1}^{\infty}))= \sum_{i=r+1}^{\infty} \dfrac{a_i}{q^i}.\label{existence_given_density_set_eq_7}
		\end{equation}
		Observe from Definition \ref{main_UDAP_set},   $\mathbb{M}(q,(b_i)_{i=1}^{\infty})\subset q^r\mathbb{N}$. So the set $\bigcup_{i=0}^{c-1}(i+q^r\mathbb{N}_0)$ and the set $(c+\mathbb{M}(q,(b_i)_{i=1}^{\infty}))$ are disjoint. Using this observation in equation \eqref{existence_given_density_set_eq_6}, we get that
		$$d(A+B)=d\left(\bigcup_{i=0}^{c-1}(i+q^r\mathbb{N}_0)\setminus\{0\}\right)+d(c+\mathbb{M}(q,(b_i)_{i=1}^{\infty}))=\dfrac{c}{q^r}+ \sum_{i=r+1}^{\infty} \dfrac{a_i}{q^i},$$
		by using equation \eqref{existence_given_density_set_eq_7}. Hence,
		by equation  \eqref{existence_given_density_set_eq_2}, we have
		$d(A+B)=\alpha$.
	\end{proof}
	\begin{proof}[Proof of Theorem \ref{Thm-3}\ref{THM-3Part1}]
		
		Let $B_1=B-\min(B)$. 	Let $k$ be a natural number such that  $B_1\subset [0,k-1]$ and $k-1=\max(B_1)$. Note that
		$\min(B_1)=0.$ Then, by Lemma \ref{existence_given_density_set},  there exists an infinite set of natural numbers $A$ such that $d(A+B_1)=\alpha$. Since density is translation invariant, we have $d(A+B)=\alpha.$
		
	\end{proof}

	\begin{proof}[Proof of Theorem \ref{Thm-3}\ref{THM-3Part2}]
		If $\alpha=1$, then the set $\mathbb{N}$ satisfies the conclusion of the theorem. So we assume  $\alpha \in (0,1)_\mathbb{R}$. Then there exists $p\in \mathbb{P}$ such that
		$\alpha<1-\frac{1}{p}$. So we can choose a nonnegative integer $k$ such that $0\leq k\leq p-2$ and
		$\frac{k}{p}\leq \alpha< \frac{k+1}{p}$.
		By applying Lemma \ref{thm_lem_1}, there exists a sequence $(a_i)_{i=1}^{\infty}$ in $[0,p-1]$ such that $a_1=k$ and
		$\alpha=\frac{k}{p}+\sum_{i=2}^{\infty}\frac{a_i}{p^i}$.
		By using Lemma \ref{density_union_of_ap}, density of the UDAP set  $\mathbb{M}(p,(a_i)_{i=1}^{\infty})$ is $\alpha$. By definition of the set $\mathbb{M}(p,(a_i)_{i=1}^{\infty})$, 
		$$\bigcup_{i=0}^{k-1}(i+p\mathbb{N}_0)\setminus\{0\}\subset\mathbb{M}(p,(a_i)_{i=1}^{\infty})\subset\bigcup_{i=0}^{k}i+p\mathbb{N}_0.$$
		
		Define $A:=(\mathbb{M}(p,(a_i)_{i=1}^{\infty})\setminus (p\mathbb{N}_0))\bigcup (k+1+p\mathbb{N}_0)$. Note that $A$ is UDAP set with density $\alpha$.
		
		Observe that $1\leq k+1\leq p-1$ and so $\gcd(k+1,p)=1$. By Dirichlet's Theorem (Lemma \ref{dirichlet}, we get  that  $k+1+p\mathbb{N}_0$ contain infinitely many primes. Therefore $A$ contains infinitely many primes.
	\end{proof}
	
	\begin{proof}[Proof of Theorem \ref{Thm-3} \ref{THM-3Part3}]
		Lemma \ref{totient_property1} provides us composite number $n\in \mathbb{N}\setminus \{1\}$ such that $\phi(n)<(1-\alpha)n$. So we can choose a nonnegative  integer $k$ such that $0\leq k\leq n-\phi(n)-1$ and
		$\frac{k}{n}\leq \alpha < \frac{k+1}{n}$.
		By applying Lemma \ref{thm_lem_1}, there exists a sequence $(a_i)_{i=2}^{\infty}$ in $[0,n-1]$  such that
		$\alpha=\frac{k}{n}+\sum_{i=2}^{\infty}\frac{a_i}{n^i}$.
		Let $T=\{a\in[1,n]: \gcd(a,n)\neq 1 \}$. Since $|T|=n-\phi(n)$, there exists a set $S$ such that $S\subset T$  and $|S|=k$. Choose 
		$$A:= (\mathbb{M}(n,(a_i)_{i=1}^{\infty}))\bigcup\left(\bigcup_{i\in S}(i+n\mathbb{N}_0)\setminus\{0\}\right), $$
		where $a_1=0$. Observe that  $A$ is UPDA set. Since $\mathbb{M}(n,(a_i)_{i=1}^{\infty})$ is subset of $n\mathbb{N}$, so it does not contain prime. We also observe that  $S+n\mathbb{N}$ does not contain prime because $\gcd(i,n)\neq 1$ for every $i\in S$. Thus  $A$ does not contain any prime.
		
		To complete the proof, we  show that density of $A$ is $\alpha$. By Corollary \ref{density_union_ap_corollary}, $$d(\mathbb{M}(n,(a_i)_{i=1}^{\infty}))=\sum_{i=2}^{\infty}\frac{a_i}{n^i}.$$ Therefore, since $\mathbb{M}(n,(a_i)_{i=1}^{\infty})$  and $S+n\mathbb{N}_0$ are disjoint sets, we get that
		$$d(A)=\frac{|S|}{n}+ \sum_{i=2}^{\infty}\dfrac{a_i}{n^i}=\frac{k}{n}+ \sum_{i=2}^{\infty}\dfrac{a_i}{n^i}=\alpha.$$
	\end{proof}

	\section{Proof of Theorem \ref{Thm-4} and Corollary \ref{THM-4Cor1} }\label{section6}
	\begin{proof}[Proof of Theorem \ref{Thm-4}]
		If $A=\mathbb{N}$, result holds for $\alpha=1$.  So we  assume  $\alpha <1$.  
		
		Suppose that there exists prime $p$ such that $g$ is primitive root modulo $p^2$ and $\gcd(a,p)=1$. By Lemma \ref{primitilive_root_property1}, $g$ is primitive root modulo $p^k$ for each positive integer $k$.  Let $m>\max\{2,\log_p(\frac{2}{\alpha})\}$. 
		Define the set $A_0$ and $B_0$ in the following way:  $$A_0:=p^m\mathbb{N}\cup (1+p^m\mathbb{N}) \quad \text{and} \quad B_0=\{ag^i: i\in [1,\phi(p^m)]\}$$
		Observe that $B_0\subset B$. This gives that	
		\begin{align}\label{THM-4eq1}
			A_0+B_0&=\left(\bigcup_{i=1}^{\phi(p^m)} ag^i+p^m\mathbb{N}\right)\bigcup\left(\bigcup_{i=1}^{\phi(p^m)} 1+ag^i+p^m\mathbb{N}\right)\subset A_0+B
		\end{align}
		Let $f: B\rightarrow [0,p^m-1]$ be homomorphism defined by $f(ag^i)=x$ where $x$ is unique nonnegative integer in $[0,p^m-1]$  such that $ag^i\equiv x\pmod{p^m}$. This implies that for every $i\in \mathbb{N}_0$,		
		$$ag^i+p^m\mathbb{N}=(f(ag^i)+p^m\mathbb{N}_0)\setminus \{f(ag^i),f(ag^i)+p^m,\cdots, ag^i\}.$$
		Using these values in equation \eqref{THM-4eq1}, we get
		\begin{equation}\label{THM-4eq2}
			A_0+B_0= \left((f(B_0)\cup (1+f(B_0)))+p^m\mathbb{N}_0\right)\setminus X,
		\end{equation}
		where $X$	is some finite set.	
		
		Since $g$ is primitive root modulo $p^m$, we get that
		$ f(B_0)=\{u\in [1,p^m-1]:p\nmid u\}$.
		This  means that $f(B_0)$ contains all those elements in $[0,p^m-1]$ that is not divisible by $p$, and $1+f(B_0)$ contains all multiples of $p$ in $[0,p^m-1]$. Thus $f(B_0)\cup (f(B_0)+1)=[0,p^m-1]$.
		Incrporating this in the equation \eqref{THM-4eq2}, we get that  $\mathbb{N}\setminus (A_0+B_0)$ is finite set. Since $B_0\subset B$, we obtain that
		\begin{equation}\label{THM-4eq4}
			|\mathbb{N}\setminus (A_0+B)|<\infty.
		\end{equation}
		
		Observe from the definition of quantity $m$ and the set $A_0$ that $d(A_0)=\frac{2}{p^m}$ and $\alpha >\frac{2}{p^m}$. Therefore to complete the proof of the theorem, we need to construct set $A$ with density $\alpha$ so that $A_0\subset A$ and $A$ is the  union of disjoint infinite arithmetic progressions.

		Since $\alpha>\frac{2}{p^m}$, we get from Lemma \ref{thm_lem_1} that there exists a sequence of integers $(a_n)_{n=1}^{\infty}$ with $0\leq a_n\leq p^m-1$ for each $n\in \mathbb{N}\setminus\{1\}$ and $2\leq a_1\leq p^m-1$ such that
		$\alpha =\sum_{n=1}^{\infty}\frac{a_n}{p^{mn}}.$
		
		Denote
		$	A:=\mathbb{M}(p^m,(a_i)_{i=1}^{\infty})$. By using Corollary \ref{density_union_ap_corollary}, we get that $d(A)=\alpha$. Since $a_1\geq 2$, we get that $A_0\subset A$ by Definition \ref{main_UDAP_set}.
		%
		Combining this with equation \eqref{THM-4eq4}, we get that		
		$\mathbb{N}\setminus (A+B)$ is finite.
	\end{proof}
	\begin{proof}[Proof of Corollary \ref{THM-4Cor1}]
		Corollary \ref{primitive_root_heath_boron_2} provides us an infinite set $P_1$, a subset of $\mathbb{P}$, such that  every element of the set $P_1$ is primitive root modulo $q$ for infinitely many prime $q$.
		Let 
		$$P_2:=\{g\in P_1: g \text{ is not primitive root  modulo } q^2 ~\forall~ q\in \mathbb{P} \}$$
		If $P_2$ is the empty set, then  for every $g\in P_1$ there exists a prime $q_g$ such that $g$ is primitive root modulo $q_g^2$. By applying Theorem \ref{Thm-4}, for every $\alpha\in (0,1]$ there exists additive complement $B_{g,\alpha}$ of the set $\{g^i:i\in \mathbb{N}\}$ such that $B_{g,\alpha}$ is union of disjoint infinite arithmetic  progressions and $d(B_{g,\alpha})=\alpha$. Hence the conclusion of Theorem \ref{Thm-4} holds for the geometric progression with initial term $g$ and common ratio $g$ for each $g\in P_1$.
		
		If $P_2$ is nonempty set, then there exists $h\in P_1$ such that $h$ is not primitive root modulo $q^2$ for every prime $q$. So there exists prime $r$ such that $h$ is  primitive root modulo $r$ and $h$ is not primitive root modulo $r^2$. By applying Lemma \ref{primitive_root_property2}, this implies that $h+tr$ is primitive root modulo $r^2$ for every positive integer $t$ with $t\not\equiv 0\pmod{r}$. Therefore conclusion of Theorem \ref{Thm-4} holds for every geometric progression with initial term $(h+tr)$ and common ratio $(h+tr)$ for every $t\in \mathbb{N}\setminus r\mathbb{N}$ by using Theorem \ref{Thm-4}.
		
		Therefore combination of the above two paragraph gives that there are infinitely many geometric progression for which conclusion of Theorem $\ref{Thm-4}$ holds.
	\end{proof}
	\section{Conclusion}
	In this article, we have focused on analysing some  special    additive complements of a set of natural numbers. These special additive complements consist of additive complement of a set disjoint to that set and additive complements which is  the union of disjoint infinite arithmetic progressions. In Theorem \ref{Thm-1} and Theorem \ref{Thm-2}, we have proved that there are infinitely many set $A$  of which additive complement always intersects that set. So this inspires some natural question, which is given below:
	\begin{question}
		Let $\alpha \in [0,1]$. Characterise all those  sets $A$, with  $d(A)=\alpha$, of which additive complement is subset of $\mathbb{N}\setminus A$?
	\end{question}
	In order to search additive complements which is UDAP set, we have proved that there are infinitely many geometric progressions for which we can get additive complements such that it is UDAP set. So one can ask existence of additive complements, which is also UDAP set, for arbitrary geometric progressions. This motivates the following question:
	\begin{question}
		Let $\alpha\in (0,1]$  and $B$ be a set of natural numbers such that its elements form a geometric progression. Does there exists an additive complements $A$ of the set $B$, with  $d(A)=\alpha$,  such that  $B$ is the union of disjoint infinite arithmetic progressions?
	\end{question}
	\noindent
	One can get affirmative solution of the above question if a set with zero density can be taken inside   UDAP sets with arbitrary density. This motivates to ask  the following question:
	\begin{question}
		Let $\alpha\in (0,1]$ and $A$ be a set of natural numbers with density $0$. Does there exist $B\subset \mathbb{N}$ such that $B\supset A$, $d(B)=\alpha$ and $B$ is the union of disjoint arithmetic progressions?
	\end{question}
	
	In context of UDAP sets, density of many  UDAP sets exist according to Lemma \ref{density_union_of_ap}. For general UDAP sets, we can ask the following question:
	\begin{question}
		Does density of an  UDAP set exist?
	\end{question}
	
	Theorem \ref{Thm-3} proves that for every finite  set $B$ and given $\alpha\in (0,1]$, there exists a set $A$ with $d(B+A)=\alpha$. In case of an infinite set $B$, existence of such $A$ is not known. So we ask the following question:
	\begin{question}
		Let $B$ be an infinite set of natural numbers and $\alpha \in [0,1]$. Does there exists a set $A$ such that $d(A+B)=\alpha$?
	\end{question}
	\noindent
	Proposition \ref{Prop-Lorentz} implies affirmative solution of the above question for the case $\alpha=1$.

	\section*{Acknowledgements}
	Authors would like to thank
	IIT Roorkee for academic and financial support.



\begin{thebibliography}{99}
		\bibitem{carather} N. L. Carothers, Real Analysis, Cambridge University Press, 2000.
		
		\bibitem{totient}M. K. Das, Pramod Eyyunni, B. R. Patil, Combinatorial properties of sparsely totient numbers, \textit{J. Ramanujan Math. Soc.} \textbf{35} (1) (2020), 1-16.
		
		\bibitem{erdos} P. Erd\H{o}s, Some results on additive number theory, \textit{Proc. Am. Math. Soc.} \textbf{5} (6) (1954), 847-853.
		
		\bibitem{Faisant} A. Faisant, G. Grekos,  R. K. Panday, and S. T. Somu, Additive complements for a given asymptotic density,  \textit{Mediterr. J. of Math.} \textbf{18} (1) (2021), 1-13.
		\bibitem{Grekos} G. Grekos, R. K. Panday, and S.T. Somu, Sumsets with prescribed lower and upper asymptotic densities, \textit{Mediterr. J. of Math.} \textbf{19} (201) (2022), 1-9.
		
		\bibitem{hardy}  G. H. Hardy and E. M. Wright, An Introduction to the Theory of Numbers, 6th edition, Oxford University Press, 2008.
		\bibitem{Heath-Brown} D. R. Heath-Brown, Artin's conjecture for primitive roots, \textit{Quart. J. Math.  Oxford} \textbf{37} (145)  (1986), 27-38.
		\bibitem{Lorentz} G. G. Lorentz, On a problem of additive number theory, \textit{Proc. Am. Math. Soc.} \textbf{5} (5) (1954), 838-841.
		
		
		\bibitem{niven} I. Niven, H. S. Zuckerman, and H. L. Montgomery, An Introduction to the Theory of Numbers, 5th edition, John Wiley and Sons, New York, 1991.
		
		
		
		
		
		
		
	\end{thebibliography}
\end{document}